\documentclass[11pt]{article}

\usepackage{amsfonts}
\usepackage{amssymb,amsmath,amsthm, enumerate}
\usepackage{latexsym}
\usepackage{fullpage}
\usepackage{hyperref}

\usepackage{color}

\newtheorem{theorem}{Theorem}[section]

\newtheorem{prop}[theorem]{Proposition}

\newtheorem{lemma}[theorem]{Lemma}
\newtheorem{cor}[theorem]{Corollary}
\newtheorem{defn}[theorem]{Definition}

\theoremstyle{definition}
\newtheorem{remark}[theorem]{Remark}

\newcounter{tenumerate}

\def \a {\alpha}
\def \b {\beta}

\def \BD {{\cal BD}}
\def \cB {{B}}

\def \e {\varepsilon}

\def \d {\delta}
\def \E {{\mathbb{E}}}
\def \ff {{\cal F}}

\def \k {\kappa}
\def \l {\lambda}

\def \P {{\mathbb{P}}}

\def \R {{\mathbb{R}}}
\def \T {{\mathbb{T}}}
\def \Var {{\rm Var}}
\def \s {\sigma}
\def \ss {{\cal S}}
\def \ee {{\cal E}}

\def \vff {{\cal VF}}
\def \Z {{\mathbb{Z}}}
\def \( {\left( }
\def\) {\right) }
\def\[ {\left[}
\def\]{\right]}

\begin{document}

\title{{\bf On the Liouville heat kernel for $k$-coarse MBRW and nonuniversality}}

\author{ Jian Ding\thanks{Partially supported by an NSF grant DMS-1455049, an Alfred Sloan fellowship, and NSF of China 11628101.}  \\ University of Chicago \and Ofer Zeitouni\thanks{Partially supported by the ERC advanced grant LogCorrelatedFields and by the
Herman P. Taubman chair at the Weizmann Institute.} \\
Weizmann Institute \\ Courant Institute \and Fuxi Zhang\thanks{Supported by NSF of China 11371040.} \\ Peking University
}

\date{}

\maketitle
\begin{abstract}
  We study the Liouville heat kernel (in the $L^2$ phase) associated with
  a class of logarithmically correlated Gaussian fields on the
  two dimensional torus. We show that for each $\e>0$ there exists such a
  field,
  whose covariance is a bounded perturbation of that of the two dimensional
  Gaussian free field, and such that the associated Liouville
  heat kernel satisfies the short time estimates,
  $$\exp \( - t^{ - \frac 1 { 1 + \frac 1 2 \gamma^2 }  - \e  } \)
\le p_t^\gamma (x, y)  \le
\exp \( - t^{-  \frac 1 { 1 + \frac 1 2 \gamma^2 }  + \e } \) ,$$
for $\gamma<1/2$.
In particular, these are different from predictions, due to Watabiki,
concerning the
Liouville heat kernel for the two dimensional Gaussian free field.
\end{abstract}

\section{Introduction}



In recent years, there has been much interest and progress in the understanding
of two dimensional
\textit{Liouville quantum gravity}, and associated processes. We do not
provide an extensive bibliography and refer instead to the original articles
and surveys
\cite{DS11, DMS,  Ber} for background. The starting point for this study is
the construction of Liouville measure, which is the exponential
of the Gaussian free field and is constructed rigorously using Kahane's
theory
of Gaussian multiplicative chaos \cite{RV14}.

One aspect that has received attention is the construction of Liouville
Brownian motion using the Liouville measure and the theory of Dirichlet forms.
Mathematically, this has been achieved in \cite{GRV13}
(see also \cite{Ber15}), and properties
of the associated Liouville heat kernel have been discussed in \cite{GRV14,
MRVZ14, AK16}.
One important motivation behind the study of the Liouville heat kernel is that it
can be used to study the geometry (and critical exponents) of
Liouville quantum gravity. Indeed,
a particularly nice application of the construction
of the Liouville heat kernel is that it allows for a clean derivation
of the so-called KPZ relations \cite{BGRV16}.
Another  important motivation, discussed
in \cite{MRVZ14}, are the predictions of Watabiki \cite{Wa} concerning
the short time behavior of the Liouville heat kernel. See the
discussion in \cite{MRVZ14, AK16} for existing (weak)
estimates on the diffusivity exponents of the Liouville heat kernel.

An important aspect of the class of logarithmically correlated Gaussian
fields (of which the 2D Gaussian free field is arguably the prominent example)
is the universality of many quantitites,
e.g. Hausdorff dimensions, statistics of the maximum, etc.,
see \cite{RV14,DRZ15}. One could naively expect that for Gaussian fields
in this class, the  predicted exponents of the Liouville heat kernel
would be universal.

Our goal in this paper is to show that this is not the case, in the sense that
the explicit predictions on Liouville heat-kernel exponents (appearing in
\cite{Wa} and discussed in \cite{MRVZ14, AK16}) do not hold for
some two dimensional logarithmically
correlated Gaussian fields which are bounded perturbations of the Gaussian free field.
Namely,
we study in this paper the heat kernel for Liouville Brownian motion
constructed with respect to a particular logarithmically
correlated field, introduced in \cite{DZ15}
under the name \textit{$k$-coarse modified branching random walk} (MBRW for short). Given $k>0$ integer,
this is the centered Gaussian field
on the torus $\T = \R^2/(4\Z)^2$, denoted $h=\{h (x) \}_{x\in \T}$,
  with covariance
 $$
G(x,y) = k \log 2 \sum_{j=0}^\infty A (x,y; 2^{-kj} ),
 $$
where $A(x, y; R) = | B(x, R) \cap B(y, R)| / |B(x,R)|$,
$B(z,R)$ is the (open) ball centered at $z$ with radius $R$ with respect to the
natural metric on the torus, and $| B |$ is the Lebesgue measure of a set $B$. The particular choice of the
scaling of the torus is not important and only done for convenience.

We will show in Section~\ref{Section.covariance} that for all $k$,
 \begin{equation} \label{Eq.covariance}
G(x,y) = \log \frac 1 { |x-y| } + \l  (|x - y|) ,
 \end{equation}
where $\l$ is continuous in $(0,2]$ and $| \l | \le 6 k$. Fixing $\gamma\in (0,2)$,
we introduce in Section~\ref{Section.LBM}, following \cite{GRV13}, the Liouville measure $\mu^\gamma$,
Liouville Brownian motion (LBM) $\{ Y_t \}$, and Liouville heat kernel
(LHK) $p^\gamma_t(x,y)$, associated with $(\gamma,h)$. Formally,
the Liouville measure on $\T$ is defined as
$\mu^\gamma (dx) : = e^{\gamma h (x) - \frac {1}{2}\gamma^2 \E h^2 (x)} d x$;
one then introduces
the positive continuous additive functional (PCAF) with respect to $\mu^\gamma$
as
$$ F (v) := \int_0^v e^{\gamma h(X_u) - \frac {\gamma^2}2 \E h (X_u)^2} d u, $$
where
$\{ X_t \}$
denotes a standard Brownian motion (SBM) on $\T$.
The LBM is then defined formally as $Y_t:=X_{F^{-1}(t)}$, and
the LHK $p_t^\gamma (x,y)$ is then the density of the Liouville semigroup with respect to $\mu^\gamma$, i.e.
 $$
E^x f(Y_t) = \int p_t^\gamma (x,y) f(y) \mu^\gamma (d y),
 $$
where the superscript $x$ is to recall that $Y_0=X_0=x$.

Let $\P$ denote the Gaussian law of $h$. The
main result of this paper is as follows.
 \begin{theorem} \label{Theorem.mainthm}
Suppose $0 \le \gamma < \frac 1 2$, and $x, y \in \T$ with $x \neq y$.
For any $\e > 0$, there exist $k (\e, x, y)$ and
a random variable $T_0$ depending on $(x,y,\gamma,k,\e,h)$ only so that
for any
$k \ge k (\e, x, y)$ and $t<T_0$,
\begin{equation}
  \label{eq-main1}
\exp \( - t^{ - \frac 1 { 1 + \frac 1 2 \gamma^2 }  - \e  } \)
\le p_t^\gamma (x, y)  \le
\exp \( - t^{-  \frac 1 { 1 + \frac 1 2 \gamma^2 }  + \e } \) \,,\quad \P \mbox{\rm -a.s.}.
 \end{equation}
 \end{theorem}
 \begin{remark}
Our result shows that the exponent of the LHK with respect to the $k$-coarse MBRW is for large $k$ and small
$\gamma$, roughly
$(1+o_k(1))/(1+\gamma^2/2)$.
In particular, it does not
match values one could guess from Watabiki's formula, see \cite{Wa,MRVZ14}, based on which one would
predict that for $\gamma$ small, the exponent is
$(1+o(\gamma))/(1+7\gamma^2/4)$.
This is yet another manifestation of the expected non-universality of exponents related to Liouville quantum gravity, across the class of logarithmically correlated Gaussian fields. See \cite{DZ15,DG16} for other examples.
  \end{remark}

\noindent
{\bf Heuristic.} We describe the strategy behind the proof of the lower bound, and the upper bound is similar. First, represent hierarchically the MBRW as follows. Let $h_j$ be independent centered Gaussian fields on $\T$ with covariance
\begin{equation}
  \label{eq-ofer1}
  \E h_j (x) h_j (y) = k \log 2 \times A(x,y; 2^{-kj}) = : g_j (x,y).
\end{equation}
Formally,
$h = \sum_{j=0}^\infty h_j$. For given $t$,
choose $r$ such that $t = 2^{-kr(1 + \frac 1 2 \gamma^2 - o (1))}$,
and decompose the field $h$ into a coarse field $\varphi_r$ and a
fine field $\psi_r$, with
 \begin{equation}
   \label{eq-2}
\varphi_r := \sum_{j=0}^{r-1} h_j , \ \ \ \psi_r := \sum_{j=r}^\infty h_j ,
 \end{equation}
with respective covariances
\begin{equation}
  \label{eq-pg-2}
G^{(1)}_r (x,y) = k \log 2 \sum_{j=0}^{r-1} A(x,y; 2^{-kj}), \  \ G^{(2)}_r (x,y) = k \log 2 \sum_{j=r}^\infty A(x,y; 2^{-kj} ).
\end{equation}
 Note that much like the MBRW, the fine field is not defined pointwise but only
 in the sense of distributions.

With $k,r$ fixed, we partition $\T$ into $2^{2 (k r+2)}$ boxes of side length $s = 2^{- k r} $,
elements of
$$\BD_r=\{[a 2^{-kr}, (a+1)2^{-kr}) \times [b 2^{-kr}, (b+1) 2^{-kr})\}_{
  a, b \in [0 , 2^{kr+2} ) \cap \Z}.$$
    We call the
    elements of $\BD_r$ \textit{$s$-boxes}.
    Similarly to \cite{DZ15},
    we will find a sequence of neighboring $s$-boxes $B_i$, $1 \le i \le I$
    (with $I \le 2^{kr(1 + \d )}$, $\d$ chosen below) connecting $x$ to $y$,
    so that the following properties (of the $B_i$'s) hold.
    The coarse field $\varphi_r$ throughout each $B_i$
    is bounded above by
    $\d kr \log 2$,
    where $\d > 0$ is small and will
    be chosen according to $\e$ in Theorem~\ref{Theorem.mainthm}.
    With probability at least $s^\d$,
    the LBM associated with the fine field $\psi_r$ crosses each $B_i$
    within time
    $s^{2 - \d }$. Forcing the
    original LBM to pass through this sequence of boxes,
    we will then conclude that
    it spends time at most
    $\le 2^{kr(1+ \d )} \times 2^{ \d \gamma k r - \frac 1 2 \gamma^2 k r } s^{2-\d} = 2^{-kr (1 + \frac 1 2 \gamma^2 - (2+\gamma)\d )} = t^{1+O(\e)}$
    crossing
    from $x$ to  the $s$-box containing $y$.
    This happens with probability at least
    $\ge ( s^\d )^{- 2^{kr (1+ \d)} }
    \ge \exp ( - t^{- \frac 1 {1+\frac 1 2 \gamma^2} + \e } ) $, and,
    modulu a localization argument, completes the proof of the lower bound.

\medskip

\noindent {\bf Structure of the paper.}
The preliminaries Section~\ref{Section.construction} is devoted to
the study of the covariance of the $k$-coarse MBRW $h$,
and in particular to verifying
that its covariance is a bounded perturbation of that of the Gaussian free field. We also discuss the power law spectrum of $h$  and the construction of the LBM with its corresponding PCAF.
In addition, Section~\ref{Section.coarsefield} is devoted to a study of
the coarse field $\varphi_r$, and results in estimates on its
fluctuations and maximum in a box.
Section~\ref{Section.finefield} is devoted to a study of the fine field; we introduce the
notions of slow and fast points/boxes and estimate related probabilities.
 (The property of being fast is used in the proof of the
 lower bound, and that of being slow is used in the upper bound.)
Finally, the proof of lower bound is contained
in Section
\ref{Section.lowerbound},
and that of upper bound is contained in Section
\ref{Section.upperbound}. Both these sections borrow crucial arguments from \cite{DZ15}.

\medskip

\noindent {\bf Notation convention}. Throughout the paper, we restrict attention to
$0 \le \gamma < 1/ 2$.  $\T$ is equipped with the natural metric inherited from the Euclidean distance.
We choose $\d>0$ small and $k$ large integer (as functions of $\e$)
and keep them fixed throughout.
We let $C_i$, $i=0,1,\dots$  be universal positive constants, independent of
all other parameters. With $r$ as described above, we let $BD_r (x)$ denote
the unique element of $\BD_r$ containing $x$. For $\ell > 0$, an $\ell$-box means a box of side length $\ell$. Let $B_\ell (x)$ denote the $\ell$-box centered at $x$, and let $B(x, \ell)$ denote
the ball centered at $x$ with radius $\ell$. For any box $B$, let
$c_B$ denote the center of $B$. If $B$ is an $\ell$-box, denote by $B^{*}$ the $(5 \ell)$-box centered at $c_B$.
We use $\P$ and $\E$ to denote the probability and expectation related to
the Gaussian field $h$. Let $P^x$ and $E^x$ be the probability and expectation related to the SBM starting at $x$. We let  $F^x$ and $F_r^x$ be the PCAFs
for the LBM and $\psi_r$-LBM started at $x$, respectively.
When the starting point $x$ needs not be emphasized,
we drop the superscript $x$.

\section{Preliminaries} \label{Section.construction}

Subsection~\ref{Section.covariance} is devoted to
the proof of \eqref{Eq.covariance}. In Subsection~\ref{Section.coarsefield},
we study  the coarse field $\varphi_r$ and bound its maximum on small boxes
as well as the fluctuation across such boxes.
Subsection~\ref{Section.LBM} is devoted to a quick review of the
construction and existence of the LBM and the LHK.

\subsection{Proof of \eqref{Eq.covariance}} \label{Section.covariance}
Let $d$ denote the $\T$ distance between $x,y$, and fix
$r_0 := r_0 (d) \ge 0$ integer
so that
 $$ 2^{- k (r_0 + 1)} < \frac d 2 \le 2^{- k r_0 } .
 $$
Denote
 $$
\theta_{j , d} : =
\arcsin(2^{kj} d/2), \quad j=0,1,\ldots, r_0. 
 $$

We compute the covariance $g_j (x,y)$, c.f. \eqref{eq-ofer1}. For $j\leq r_0$, note that
$R: = 2^{- k j} \ge \frac d 2$; set $\theta=\theta_{j,d}$. Then $|B(x,R) \cap B(y,R)| = (\pi - 2 \theta) R^2 - 2 R^2 \sin (\theta) \cos (\theta) = \pi R^2 - R^2 (2 \theta + \sin (2 \theta))$, which implies that $A(x,y; R) = 1 - \frac 1  {\pi} (2 \theta + \sin (2 \theta) )$. It follows that with $j\in \Z_+$,
 \begin{equation} \label{Eq.COVcovhj}
g_j (x,y) = \left\{ \begin{array}{ll}  k \log 2 - \frac {k \log 2} \pi
  \big( 2 \theta_{j , d} +  \sin (2 \theta_{j , d} ) \big),
  & \mbox{if } j \le r_0,  \\ 0, & \mbox{otherwise.} \end{array} \right.
 \end{equation}
We now write
\begin{equation}
  \label{eq-ofer2}
  G(x,y)=\sum_{j=0}^\infty g_j(x,y)=\sum_{j=0}^{r_0} g_j(x,y)=
k\log 2\left( (r_0+1)-\frac1\pi \sum_{j=0}^{r_0} \big( 2\theta_{j,d}+
\sin(2\theta_{j,d}) \big)\right)\,.
\end{equation}
Since $r_0=r_0(d)$, we obtain that $G(x,y)=g(d)$ for some function $g:(0,2]
\to \R_+$. We now show that $g$ is continuous. Indeed, note that for any fixed
$j$,   $d\mapsto  \theta_{j,d}$ is continuous (in $d \in [0, 2^{1-kj}]$). Thus the only possible discontinuities of $g$ on $(0,2]$ are whenever $-\log_2(d/2)/k$ is an integer ({\em i.e.} equals $r_0(d)$); however, for such
$d$ we obtain that $\theta_{r_0(d),d}=\pi/2$, which together with the continuity of $d\mapsto \theta_{j,d}$, yields the continuity of $g$.

 To estimate $g(d)$, note that
for all $\theta \in [0, \frac \pi 2]$,
$0 \le \sin (2 \theta) \le 2\sin(\theta)$ and $\theta\leq 2\sin(\theta)$,
 and therefore
 \begin{equation} \label{Eq.COVtheta}
0 \le 2 \theta + \sin (2 \theta) \le  6\sin(\theta).
 \end{equation}
 In particular,
 $$
\frac1\pi |\sum_{j=0}^{r_0}(2\theta_{j,d}+ \sin(2\theta_{j,d}))| \leq \frac6\pi \sum_{j=0}^{r_0} 2^{ - k (r_0 - j)} \leq \frac6\pi \sum_{i=0}^\infty 2^{-ki} \leq \frac {12} \pi \le 4.
 $$
On the other hand, $|k(r_0+1)\log 2 + \log d|\leq (k+1)\log 2\leq 2 k$.
Combining the last two displays with \eqref{eq-ofer2} shows that
 $$
|g(d)+\log d|\leq 6 k \,,
 $$
yielding \eqref{Eq.covariance}.

\subsection{The coarse field}  \label{Section.coarsefield}
Note that $g_j(x,y)$ is a positive definite kernel on $L^2(\T)$,
since, with $R=R_j=2^{-{kj}}$,
$$\hat g_j(x,y)=|B(0,R)|g_j(x,y)=
\int_{\T} dz \ {\bf 1}_{|z - x|\leq R} {\bf 1}_{|z-y|\leq R}$$
and therefore,
for any $f\in L^2(\T)$,
 $$
\int_{(\T)^2} f(x)f(y) \hat g_j(x,y) dx dy = \int_{\T} dz \left(\int_{\T} dx \ f(x) {\bf 1}_{|x-z|\leq R}
\right)^2 \geq 0\,.
 $$
Since $g_j(x,y)$ is Lipshitz continuous, Kolmogorov's criterion implies that
the associated Gaussian field $x\mapsto h_j(x)$ is continuous almost surely (more precisely, there exists a version of the field which is continuous almost surely).
Consequently, the coarse field $\varphi_r$ is also smooth.
In this subsection, we estimate the maximum value as well
as the fluctuations of $\varphi_r$ in a box.

We begin by recalling an easy consequence of Dudley's criterion.
\begin{lemma}\label{Lem.maxinbox}
(\cite[Theorem 4.1]{A90}) Let $B \subset \mathbb{Z}^2$ be a box of side length $\ell$ and $\{ \eta_w : w \in B \}$ be a mean zero Gaussian field satisfying
 $$
\E (\eta_z - \eta_w)^2 \le |z-w|_\infty / \ell \ \ \ {\rm for \ all \ } z,w \in B.
 $$
Then $\E \max_{w \in B} \eta_w \le C_0$, where $C_0$ is a universal constant.
 \end{lemma}
 The next lemma is usually referred to as the Borell, or
 Ibragimov-Sudakov-Tsirelson, inequality. See,  e.g., \cite[(7.4), (2.26)]{L01} as well as discussions in \cite[Page 61]{L01}.
 \begin{lemma} \label{Lem.concentration}
Let $\{ \eta_z : z \in B \}$ be a Gaussian field on a finite index set $B$. Set $\s^2 = \max_{z \in B} \Var (\eta_z)$. Then for all $\lambda, a > 0$,
 $$
\E[ \exp\{\lambda (\max_{z \in B} \eta_z - \E \max_{z \in B} \eta_z)\} ]\leq e^{\frac{\lambda^2 \sigma^2}{2}} , \ \mbox{ and } \ \P (| \max_{z \in B} \eta_z - \E \max_{z \in B} \eta_z | \ge a) \le 2 e^{-\frac {a^2}{2 \s^2}} .
 $$

 \end{lemma}

 \begin{prop} \label{Prop.smooth}
Suppose $k$ is large. For all $r \ge 1$,
 $$
\E (\varphi_r (x) - \varphi_r (y))^2 \le 2^{kr} |x-y|, \ \ \ \forall x, y \in \T.
 $$
  \end{prop}
 \begin{proof}
Use the notation in Subsection~\ref{Section.covariance}. Let $d = |x-y|$, $r_0 = r_0 (d)$. By \eqref{Eq.COVcovhj} and \eqref{Eq.COVtheta},
 $$
\E (h_j (x) - h_j (y))^2 =
\frac {2k \log 2} \pi
  \big( 2 \theta_{j , d} +  \sin (2 \theta_{j , d} ))
  \le \left\{ \begin{array}{ll} 2 k d 2^{k j} , &  \forall j \le r_0 , \\ 2 k , & \forall j > r_0 , \end{array} \right.
 $$
where we use $\sin (\theta_{j,d}) = 2^{kj} d / 2$ in the case $j \le r_0$.

If $r_0 \ge r-1 $,
 $$
\E (\varphi_r (x) - \varphi_r (y))^2 = \sum_{j =0}^{r-1} \E (h_j (x) - h_j (y))^2 \le 2kd \sum_{j =0}^{r-1} 2^{kj} \le 2^{kr} d.
 $$
Otherwise, $r_0 \le r - 2$.
 $$
\E (\varphi_r (x) - \varphi_r (y))^2 = 2 k (r - r_0 - 1) + \sum_{j =0}^{r_0} 2 k d 2^{kj} \le 2 k (r - r_0 - 1) + 4kd 2^{k r_0} .
 $$
Note $2^{kr} d \ge 2^{k(r-r_0 -1) + 1}$ and $r - r_0 - 1 \ge 1$. It follows that
 $$
\E (\varphi_r (x) - \varphi_r (y))^2 \le \frac {k(r - r_0 - 1)} {2^{k (r-r_0 - 1)}} 2^{kr} d  +    \frac {4k}{ 2^{k (r-r_0)}} 2^{kr} d \le 2^{kr} d ,
 $$
since $k$ is large enough.
\end{proof}

 \begin{cor} \label{Cor.maxvalue}
Suppose $k$ is large. Let $B$ denote a box of side length $\ell$, and set $M : = \max_{z \in B} \varphi_r (z)$. Then, $\E M \le \sqrt 2 C_0 \sqrt {2^{kr} \ell }$.
 \end{cor}
 \begin{proof}
We discretize $B$ by dividing $B$ into $2^{2n}$ identical boxes $\tilde B$'s
and identifying the lower left corner $\tilde c$ of each
$\tilde B$ as a point in $\Z^2$.
Denote by $M_n$ the maximum value of $\varphi_r$ over these $\tilde c$'s.
By the continuity of the coarse field,
$M_n$ increases to $M$ as $n \to \infty$.
By Proposition~\ref{Prop.smooth}, we can apply Lemma~\ref{Lem.maxinbox} to $\varphi_r / \sqrt {2^{kr} 2 \ell}$ and conclude that
$\E M_n \le \sqrt 2 C_0 \sqrt {2^{k r} \ell}$. The monotone convergence theorem
yields the result.
 \end{proof}
 \begin{cor} \label{Cor.fluctuation}
There exist $r_0 = r_0 (k, \d)$ such that the following holds for $k$ large and $r \ge r_0$. Enumerate the boxes in $\BD_r$ arbitrarily as $B_i$, $i =1, \ldots, 2^{2(kr+2)}$. Denote $M_i = \max_{x \in B_i^{*}} \varphi_r (x)$, $M_i^f = \sup_{x \in B_i^{*}} |\varphi_r (x) - \varphi_r (c_{B_i})|$, and
$M^f = \max_{1 \le i \le 2^{2 (k r + 2 ) }} M_i^f$. Then
 $$
\P (M_i \ge \d k r \log 2) \le 2 e^{- \frac 1 8 \d^2 k r \log 2}, \ \ \ \P (M^f \ge \d k r \log 2) \le e^{- r}.
 $$
 \end{cor}
 \begin{proof}
Note that, for all $x$,
$\E \varphi_r (x)^2 = k r \log 2$.  By  Corollary~\ref{Cor.maxvalue}, $\E M_i \le \sqrt 2 C_0 \sqrt {5} \le \frac 1 2 \d k r \log 2$ for $r \ge r_0 (k, \d)$. By Lemma~\ref{Lem.concentration},
 $$
\P ( M_i \ge \d k r \log 2) \le \P (M_i - \E M_i \ge \frac 1 2 \d k r \log 2) \le 2 e^{- (\frac 1 2 \d kr \log 2 )^2 / (2 k r \log 2) } = 2 e^{- \frac 1 8 \d^2 k r \log 2} .
 $$

Denote $\hat M_i^f := \sup_{x \in B_i^{*}} (\varphi_r (x) - \varphi_r (c_{B_i}))$. Similarly, we have $\P (\hat M_i^f \ge \d k r \log 2) \le 2 e^{- \frac 1 {32} (\d k r \log 2)^2 }$, noting $\E \hat M^f_i = \E M_i$ and  by Proposition~\ref{Prop.smooth}, $\E (\varphi_r (x) - \varphi_r (c_{B_i} ))^2 \le 2^{k r} |x - c_{B_i}  | \le 4$ for all $x \in B_i^{*}$. Furthermore, by a union bound and symmetry,
 $$
\P (M^f \ge \d k r \log 2) \le \sum_{i=1}^{2^{2 (k r + 2) }} 2 \P (\hat M^f_i \ge \d k r \log 2) \le 64 \times 2^{2 k r}  e^{- \frac {( \d k \log 2 )^2} {32} r^2 } \le e^{- r},
 $$
where in the last inequality we use $r \ge r_0 (k, \d)$.
 \end{proof}

\subsection{Construction of the LBM and LHK} \label{Section.LBM}

There are several ways to construct the Liouville measure $\mu^\gamma$ with respect to $h$, say, via the method of Gaussian multiplicative chaos \cite{Kahane85}. In our case, since we deal with $\gamma<1/2$, it is particulaly simple since $L^2$ methods apply. So, in the rest of this section we concentrate on the construction of the LBM and LHK.

Suppose $\e = 2^{-k r}$. Then,
 \begin{equation} \label{Eq.selfsym}
G(x,y) = G_r^{(2)} (\e x, \e y), \ \ \  \mbox{\em i.e. } G(\e x, \e y)  = G(x,y) + G_r^{(1)} (\e x, \e y)
 \end{equation}
since $A(\e x, \e y; 2^{- k (r+ j)}) = A (x, y ; 2^{- k j})$. By \eqref{Eq.COVcovhj},
 $$
G_r^{(1)} (\e x, \e y) \le G_r^{(1)} (\e x, \e x) = k r \log 2 = \log \frac 1 \e \ .
 $$
It follows that
 \begin{equation} \label{Eq.COVepsiloncov}
G(\e x, \e y) \le G(x,y)  + \log \frac 1 \e \ .
 \end{equation}

Let $\Omega_\e$ be a Gaussian field independent of $h$, with $\E \Omega_\e = 0$ and $\E \Omega_\e (x) \Omega_\e (y) = G_r^{(1)} (\e x, \e y)$. Actually, $\Omega_\e$ is a copy of the coarse field $\varphi_r$ if we regard $x$ as $\e x$.
Then
 $$
\{ h (\e x) \}_x \stackrel d = \{ h (x) + \Omega_\e (x) \}_x \ , \ \ \  \{ \Omega_\e (x) \}_x \stackrel d = \{ \varphi_r (\e x) \}_x \ .
 $$
Let $M = \max_{x \in [-1,1]^2} \Omega_\e (x)$. It follows that for $q \in [0, 4 / \gamma^2]$,
 $$
\E \mu^\gamma (B(0,\e))^q \le \e^{(2 + \frac 1 2 \gamma^2) q} \E e^{\gamma q M} \E \mu^\gamma (B(0,1))^q .
 $$
Note $M \stackrel d = \max_{x \in [-\e, \e]^2 } \varphi_r (x)$. By Lemma~\ref{Lem.concentration} and Corollary~\ref{Cor.maxvalue},  $\E e^{\gamma q M} \le \tilde C (q) \e^{ - \frac 1 2 \gamma^2 q^2}$ , where $\tilde C (q)$ is a constant depending on $q$ (as well as $\gamma$).  Thus
 $$
\E \mu^\gamma (B(0,\e))^q \le \hat C(q) \e^{\xi (q)} ,
 $$
where $\hat C (q) = \tilde C (q) \E \mu^\gamma (B(0,1))^q$, and
 $$
\xi (q) = (2 + \frac {\gamma^2} 2) q - \frac {\gamma^2} 2 q^2 .
 $$
For any $2^{-k (r + 1)} < \e \le 2^{-k r}$, we take $C(q) = \hat C(q) 2^{- k \xi (q)}$ and conclude that
 \begin{equation} \label{Eq.COVpowelaw}
\E  \mu^\gamma (B(0,\e))^q \le \E  \mu^\gamma (B(0, 2^{-kr}))^q  \le \hat C (q)  2^{- k r \xi (q)} \le C(q) \e^{\xi (q)} .
 \end{equation}

Recall that the coarse field $\varphi_r$ is smooth, so
 $$
H_r (u) : = \int_0^u e^{\gamma \varphi_r (X_v) - \frac 1 2 \gamma^2 \E \varphi_r (X_v)^2} d v
 $$
is well-defined.

With \eqref{Eq.COVepsiloncov} and \eqref{Eq.COVpowelaw},
one can follow the arguments in \cite[Section 2]{GRV13}
and obtain the following conclusions.
Let $F$ denote the PCAF associated with $\mu^\gamma$. Then,
$\P$-a.s., the limit of $H_r$ in $P^x$-probability exists
and it is the PCAF $F$; that is,
$P^x(\sup_{0 \le t \le T} |F(u) - H_r (u) |> a ) \to_{r\to\infty} 0$, for all $a > 0$ and $T > 0$. Further,
the process
$Y_t : = X_{F^{-1} (t)}$ is a strong Markov process, which
is called the LBM with respect to $\mu^\gamma$.
The LHK $p^\gamma_t (x,y)$ exists and satisfies $E^x f(Y_t) =
\int f (y)  p_t (x,y)  \mu^\gamma (d y)$. Furthermore, by \cite[Theorem 0.1]{GRV14}
and parallel arguments in \cite{MRVZ14},
$p^\gamma_t (x,y)$ is continuous in $(t,x,y)$.

\section{Fast/slow points/boxes of the fine field} \label{Section.finefield}

This section is devoted to the study of
properties  of the fine field.
For the lower bound on the LHK, we need to construct regions which are fast
to cross for the LBM,
while for the upper bound we will need to create obstacles, i.e.
regions which force the LBM to be slow. Toward this end,
we introduce in Definitions
\ref{def-fast} and \ref{def-slow}
the notions of fast/slow points and boxes, and
estimate,
in Lemma~\ref{Lemma.fastpoint} and \ref{Lemma.fastbox},
the probability that a point/box is fast/slow.

Throughout, we fix $s=2^{-kr}$ for an appropriate integer $r \ge 1$ (as
explained in the introduction, $r$, and hence $s$, are chosen so that
$t=s^{1+\frac12 \gamma^2+o(1)}$). This choice determines the
fine field $\psi_r$, see \eqref{eq-2}.
With this choice, one can construct the PCAF $F_r$ based on $\psi_r$
in the same way as $F$ was constructed, by replacing the measure $\mu^\gamma$ with the
truncated measure $\mu_r^\gamma$ written formally as $\mu_r^\gamma(dx)=
e^{\gamma \psi_r(x)-\frac{\gamma^2}{2} \E \psi_r^2(x)} dx$
(as before, the actual construction involves  the smooth cutoff
$\psi_{r,w}:=
  \sum_{j=r}^w h_j$ and taking the limit as $w\to\infty$). Formally,
  we write
  \begin{equation}
    \label{eq-11a} F_r (v) = \int_0^v
    e^{\gamma \psi_r  (X_u) - \frac 1 2 \gamma^2 \E \psi_r  (X_u)^2 } d u .
  \end{equation}
  We note also that the
  sequence of approximating PCAF
  $$F_{r,w}(v)
  : = \int_0^v e^{\gamma \psi_{r,w} (X_u) -
  \frac 1 2 \gamma^2 \E \psi_{r,w} (X_u)^2} d u
 $$
 converges as $w\to\infty$,
 in the sense described at the end of Section \ref{Section.construction}, to $F_r$.

Fix
$ \d_1, \d_2, \d_3, \e_1, \e_2, \e_3>0$ small, possibly depending on $k,\gamma$ and $s$.
Fix $z\in \T$ and recall that $B_\ell (z)$ denotes the $\ell$-box centered at $z$.
Let $\s_{z, \ell}$ denote the time that the SBM (starting from $z$) hits $\partial B_\ell (z)$.

\begin{defn}[Fast points and boxes]
  \label{def-fast}
A point $z$ is said to be {\em fast} if
\begin{equation}
  \label{eq-pg7}
  P^z (F_r (s^2 \wedge \s_{z, 6s} ) \le s^2/ \d_1) \ge 1 - \d_2.
\end{equation}
The set of fast points is denoted by $\ff$.
An $s$-box $B$ is said to be {\em fast} if $|B \cap \ff | \ge \d_3 s^2$.
\end{defn}

\begin{defn}[Slow points and boxes]
  \label{def-slow}
A point $z$ is said to be {\em slow} if
\begin{equation}
  \label{eq-pg7a}
  P^z (F_r (\s_{z,s} ) \ge \e_1 s^2 ) \ge \e_2.
\end{equation}
The set of slow points is denoted by
$\ss$. An $s$-box $B$ is said to be {\em slow} if
$|B \cap \ss | \ge \e_3 s^2$.
\end{defn}
We emphasize that the notions of fast/slow points and boxes depend on the
fine field $\psi_r$ only. Further, a point (or box) may be
fast and slow simultaneously.

Our fundamental estimate concerning fast/slow points is contained in the next lemma.
 \begin{lemma} \label{Lemma.fastpoint}
There exist universal positive constants $C_1, C_2, C_3$
such that the following hold.\\
{\rm (i)} $\P (z \in \ff) \ge 1 - C_1 \frac {\d_1} {\d_2} $. \\
{\rm (ii)} For $\e_1 \le C_2$ and $\e_2 \le C_3 e^{- 6 k \gamma^2}$,
we have $\P (z \in \ss) \ge 120 C_3 e^{- 6 k \gamma^2}$.
 \end{lemma}
 \begin{proof}
(i) Set $\xi = F_r^z (s^2 \wedge \s_{z,6s} ) $ and $\eta = P^z ( \xi > s^2 / \d_1 )$. By definition,
\begin{equation}
  \label{eq-pg7b}
  \P (z \notin \ff ) = \P (\eta > \d_2) \le \E \eta / \d_2.
\end{equation}
Note that
 $$
\E \eta = E^z \P (\xi > s^2 / \d_1 ) \le \frac{\d_1} {s^2} E^z \E \xi = \frac{\d_1} {s^2} E^z (s^2 \wedge \s_{z,6s} ).
 $$
Define $C_1 : = E^0 ( 1 \wedge \s_6 )$, where $\s_6$ is
the time that the SBM in $\R^2$ hits the boundary of $[-3,3]^2$.
Then, by scale invariance of Brownian motion,
$E^z (s^2 \wedge \s_{z, 6s} ) = C_1 s^2$.
Combining the last two displays with
\eqref{eq-pg7b}, one obtains
$\P (z \notin \ff ) \le C_1 \d_1 / \d_2$, completing the proof.

(ii) We use the abbreviation $\s = \s_{z,s}$ and set now
$\xi =  F_r^z (\s) $ and  $\eta = P^z (\xi \geq \e_1 s^2)$.
Without loss of generality, we suppose $z = (0, 0)$ and consistently
drop $z$ from the notation, writing $B_s=B_s(z)$.
Since $\eta \le 1$, we have
$\E \eta = \E \eta 1_{\eta \ge \e_2} +
\E \eta 1_{\eta < \e_2} \le \P (\eta \ge \e_2) + \e_2$. By definition,
 \begin{equation} \label{Eq.zisslow}
\P  ((0,0) \in \ss)  = \P (\eta \ge \e_2 ) \ge
\E \eta -  \e_2 = E \P (\xi \ge \e_1 s^2) - \e_2 .
 \end{equation}
We are going to estimate $ \P (\xi \ge \e_1 s^2) $
via the second moment method. Recall
that
$\E \xi = \s $, which has order $s^2$.  To compute the second moment,
note that since $\gamma<1/2$, the sequence of squares of
approximating
PCAFs $(F_{r,w})^2$ are uniformly (in $w$) integrable (see the argument just after
\eqref{eq-pg-8} below)
and therefore
%
 \begin{eqnarray*}
 \E \xi^2
  & = &
\E F_r (\s)^2 = \int_0^\s \int_0^\s \E e^{\gamma \psi_r  (X_u ) - \frac 1 2 \gamma^2 \E \psi_r  (X_u )^2 + \gamma \psi_r  (X_v ) - \frac 1 2 \E \psi_r  (X_v )^2} d u d v
  \\ & = &
\int_0^\s \int_0^\s  e^{\gamma^2 G_r^{(2)} (X_u , X_v  ) }  d u d v=
\int_{w, w^\prime \in B_s } e^{\gamma^2 G_r^{(2)} (w, w^\prime) } \nu (dw) \nu (dw^\prime) = : I_{\gamma^2},
 \end{eqnarray*}
 where $\{ X_u \}$ is the SBM starting from $(0,0)$,
 $G_r^{(2)}$ is defined in
 \eqref{eq-pg-2}, and
 $\nu$
denotes the occupation measure of $\{ X_u \}$ before exiting $B_s$,
i.e.
 $$
\int_{w \in B_s } f(w) \nu (dw) = \int_0^\s f (X_u ) d u .
 $$
Let $\hat w = 2^{kr} w$ and $\hat w^\prime = 2^{kr} w^\prime$,
with $\hat w,\hat w^\prime\in \T$.
By \eqref{Eq.covariance} and \eqref{Eq.selfsym},
$$G_r^{(2)} (w, w^\prime) = G(\hat w, \hat w^\prime) \le \log \frac 1 {|\hat w - \hat w^\prime|} +  6 k = \log \frac s {|w-w^\prime|} + 6k. $$
Consequently,
 $$
I_{\gamma^2} \le e^{6 k \gamma^2} s^{\gamma^2} \int_{w, w^\prime \in B_s} \frac 1 {|w- w^\prime|^{\gamma^2}} \nu (d w) \nu (d w^\prime) = e^{6 k \gamma^2} s^{\gamma^2} \int_0^\s \int_0^\s \frac 1 {|X_u  - X_v  |^{\gamma^2}} d u d v .
 $$
Let $\hat X_u  = \frac 1 s X_{s^2 u} $, and let
$\hat \s = \s / {s^2}$ be the time that the
SBM $\{ \hat X \}$ started at $(0,0)$ exits $[-1/2,1/2]^2$. Then
 $$
I_{\gamma^2} \le e^{6 k \gamma^2} s^4 \int_0^{\hat \s} \int_0^{\hat \s} \frac 1 {|\hat X_u  - \hat X_v |^{\gamma^2}} d u d v  .
 $$
Note $|\frac {\hat X_u - \hat X_v} {\sqrt 2}|^{\gamma^2} \ge |\frac {\hat X_u - \hat X_v} {\sqrt 2}|^{1/4}$, since $|\hat X_u - \hat X_v| \le \sqrt 2$ and $\gamma^2 \le 1/4$. Thus,
 $$
|\hat X_u - \hat X_v |^{\gamma^2} \ge \frac 1 2 | \hat X_u - \hat X_v |^{1/4} .
 $$
 It follows that
 \begin{equation}
   \label{eq-pg-8}
I_{\gamma^2} \le 2 e^{6 k \gamma^2} s^4 \hat I, \ \ \ \mbox{where }  \hat I = \int_0^{\hat \s} \int_0^{\hat \s} \frac 1 {|\hat X_u - \hat X_v|^{1/4}} d u d v .
\end{equation}
Note that
$\hat I $ is a random variable depending only on
the SBM $\{ \hat X \}$. By \cite[Theorem 4.33]{MP10},
$E \hat I < \infty$.  Consequently, there exists a
universal constant $\tilde C_1$ such that
$P( \hat I \le \frac 1 2 \tilde C_1) \ge 3/4$. Hence, the event $E_1 :
= \{ \E \xi^2 \le \tilde C_1 e^{6 k \gamma^2} s^4 \}$ has probability $P (E_1) \ge 3/4$. By the
scaling invariance of the SBM, there exists a universal positive constant
$C_2$ such that the event $E_2 = \{ \s \ge 2 C_2 s^2 \}$ has probability $\ge 3/4$. Thus, $P (E_1\cap  E_2) \ge 1/4$.

Assume $E_1\cap  E_2$ happens. On the one hand, on $E_1$,
 $$
\P ( \xi \ge \e_1 s^2) \ge \frac { \( \E \xi 1_{\xi \ge \e_1 s^2} \) ^2 }{\E \xi^2} \ge \frac 1 {\tilde C_1 e^{6 k \gamma^2} s^4} \( \E \xi 1_{\xi \ge \e_1 s^2} \) ^2  .
 $$
On the other hand, on $E_2$, $\xi = F_r (\s) \ge  F_r (2 C_2 s^2 ) = : \zeta$.
Note that
$2 C_2 s^2 = \E \zeta \le \E \zeta 1_{\zeta \ge \e_1 s^2} + \e_1 s^2$.
We have $\E \xi 1_{\xi \ge \e_1 s^2}  \ge \E \zeta 1_{\zeta \ge \e_1 s^2} \ge (2 C_2 - \e_1) s^2 \ge C_2 s^2$, where we use the assumption $\e_1 \le C_2$. Thus,
 $$
\P ( \xi \ge \e_1 s^2) \ge \frac {\( C_2 {s^2} \) ^2 } {\tilde C_1 e^{6 k \gamma^2} s^4 } = \frac {C_2^2} {\tilde C_1 } e^{- 6 k \gamma^2} , \ \ \ \mbox{on }
E_1 \cap E_2.
 $$
Consequently,
 $$
E \P (\xi \ge  \e_1 s^2) \ge E \(  \P ( \xi \ge \e_1 s^2) {\bf 1}_{E_1 \cap
E_2} \)  \ge  \frac {C_2^2} {\tilde C_1 } e^{- 6 k \gamma^2} \times P (E_1\cap
E_2) \ge \frac {C_2^2} {4 \tilde C_1 } e^{- 6 k \gamma^2} .
 $$
Take $C_3 : = C_2^2 / (484 \tilde C_1)$. Then $E \P (\xi \ge  \e_1 s^2)
\ge 121  C_3 e^{- 6 k \gamma^2}$. This, together with \eqref{Eq.zisslow}
and the assumption $\e_2 \le C_3 e^{- 6 k \gamma^2}$, implies the result.
 \end{proof}

The next lemma estimates the probability that an $s$-box $B$ is fast/slow.
 \begin{lemma}\label{Lemma.fastbox}
{\rm (i)} $\P (B \mbox{ is fast}) \ge 1 - C_1 \frac {\d_1} {\d_2} - \d_3$. \\
{\rm (ii)}  Suppose $\e_2 \le C_3 e^{- 6 k \gamma^2}$ and $\e_3 \le C_3^2 e^{- 12 k \gamma^2}$. Then, $\P (B \mbox{ is slow}) \ge 1 -  \e_1^{C_3 e^{-6k\gamma^2} 2^{- 2k}}$ if $\e_1$ is less than some constant $\e_1 (\gamma, k)$.
 \end{lemma}
 \begin{proof}
(i) By Lemma~\ref{Lemma.fastpoint}(i) and the translation invariance of the fine field $\psi_r$, $\E |B \cap \ff | \ge ( 1 - C_1 \frac {\d_1} {\d_2} ) s^2$.  Since $| B \cap \ff  | \le |B| \le s^2$, $| B \cap \ff  | \le | B \cap \ff  | {\bf 1}_{ |B \cap \ff  | < \d_3 s^2} + | B \cap \ff  | {\bf 1}_{ |B \cap \ff  | \ge \d_3 s^2 }  \le \d_3 s^2 + s^2 {\bf 1}_{|B \cap \ff  | \ge \d_3 s^2 } $. Hence, $\E |B \cap \ff | - \d_3 s^2 \le s^2 \P (|B \cap \ff | \ge \d_3 s^2 ) = s^2 \P (B \mbox{ is fast})$. Therefore, $\P (B \mbox{ is fast}) \ge \frac 1 {s^2} \( \E | B \cap \ff  | - \d_3 s^2 \) \ge 1 - C_1 \frac {\d_1} {\d_2} - \d_3$.

\medskip

(ii) Our strategy is as follows. We will divide $B$ into
$n^2$ identical  boxes $\tilde B$ of side length $\tilde s = s / n$,
where $n$ is to be chosen properly to support the following arguments.
In each box
$\tilde B$, one can find $O(s^2 / n^2)$ slow points in average,
by Lemma~\ref{Lemma.fastpoint}(ii). Then, we would like
to use large deviations to show that, with high probability,
there are at least $\d_3 s^2$ slow points in $B$,
i.e. $B$ is slow. Unfortunately, the random variables
$|\tilde B \cap \ss |$'s, measuring the size of the cluster of
slow points in the smaller boxes $\tilde B$, are heavily dependent.
To obtain the
appropriate large deviation estimates by independence,
we will replace $\s_{z, s}$ in \eqref{eq-pg7a} by
$\s_{z, \tilde s} $, and use a  new parameters $\tilde \e_1$
to
define the property of a point to be
$\widetilde{slow}$.
Let $\tilde S$ consist of $\widetilde{\mbox{slow}}$ points.
Then, the random variables
$|B_i \cap \tilde \ss |$'s are almost independent, and good large
deviation estimates for their sums can be obtained.
Finally, we will show that by choosing $\tilde \e_1$
properly, $B \cap \tilde S \subset B \cap S$ with high probability,
completing the proof.

The actual
proof is in four steps. In the first step, we set the parameters $n$ and
$\tilde \e_1$, and give the definition of being $\widetilde{\mbox{slow}}$. In the second step, we will show $|B \cap \tilde S| \ge \d_3 s^2$ with high probability. In the third step, we will show $B \cap \tilde S \subset B \cap S$ with high probability. In the last step, we collect the results obtained and show (ii).

\medskip

\noindent {\bf Step 1.} Let
 \begin{equation} \label{Eq.defofn}
\k := \sqrt {- \log \e_1} , \ \ \ r_0 := \lfloor \frac 1 k \log_2 \k \rfloor , \ \ \ n := 2^{k r_0}.
 \end{equation}
Equivalently, we write $\e_1$ in the form of $e^{- \k^2}$, pick $r_0$ such that $2^{k r_0} \le \k < 2^{k (r_0 + 1)}$, and set $n = 2^{k r_0}$. Take
 \begin{equation} \label{Eq.tildeepsl}
\tilde \e_1 =  n^{2 \gamma n  + \frac {\gamma^2} 2  + 2 } \e_1 .
 \end{equation}
The parameters $n$ and $\tilde \e_1$ depend only on $\e_1$ (and $k$,$\gamma$). As $\e_1 \to 0$, we have $\k \to \infty$, and $r_0 \to \infty$ as well as $n \to \infty$. Furthermore, $\tilde \e_1 \to 0$, since $\tilde \e_1 \le e^{(2 \gamma n  + \gamma^2 / 2 + 2) \log n} e^{ - \k^2} \le  e^{(2 \gamma n  + \gamma^2 / 2 + 2) \log n - n^2} $ and $n \to \infty$. Therefore, there exists a constant $\e_1 (\gamma, k)$ such that $\tilde \e_1 \le C_2$ if $ \e_1 \le \e_1 (\gamma, k)$.  Furthermore, we pick $\e_1 (\gamma, k)$ such that
 \begin{equation} \label{Eq.largenslowbox}
2 e^{- \frac {(2 n \log n - 2 C_0 \sqrt {n} )^2} {2 \log n}} \le  e^{- n^2 \log n }  , \ \ \  e^{ - 2 C_3 e^{- 6 k \gamma^2} n^2 } + e^{ - n^2  \log n }  \le e^{  - C_3 e^{- 6 k \gamma^2} n^2 }
 \end{equation}
as $\e_1 \le \e_1 (\gamma, k)$. Note that $\tilde \e_1$ and $\e_2$ satisfy
the assumptions in Lemma~\ref{Lemma.fastpoint}(ii) for $\e_1$ and $\e_2$.

Let $\tilde s := s / n$, and $\tilde r : = r+r_0$ such that $\tilde s = 2^{- k \tilde r}$. We say that
 $$
\mbox{a point $z$ is $\widetilde{\mbox{slow}}$ if } P^z (F_{\tilde r} (\s_{z, \tilde s} ) \ge \tilde \e_1 \tilde s^2 ) \ge  \e_2 .
 $$
Denote by $\tilde \ss$ the set of $\widetilde{\mbox{slow}}$ points.

\medskip

\noindent {\bf Step 2.} Suppose $\tilde B$ is an $\tilde s$-box.
Applying Lemma~\ref{Lemma.fastpoint}(ii) to the $\widetilde{\mbox{slow}}$ points, we obtain
$\E | \tilde B \cap \tilde \ss | \ge 120 C_3 e^{- 6 k \gamma^2}  \tilde s^2 = 2 a \tilde s^2$, where we denote
 \begin{equation} \label{Eq.aforslowpoint}
a = 60 C_3 e^{- 6 k \gamma^2} .
 \end{equation}
Note that
$|\tilde B \cap \tilde \ss| \le \tilde s^2$, which implies that $ \E |\tilde B \cap \tilde \ss | =  \E |\tilde B \cap \tilde \ss | {\bf 1}_{ |\tilde B \cap \tilde \ss | \ge a \tilde s^2}  + \E |\tilde B \cap \tilde \ss | {\bf 1}_{ |\tilde B \cap \tilde \ss | < a \tilde s^2 }  \le \tilde s^2 \P ( |\tilde B \cap \tilde \ss | \ge a \tilde s^2 ) + a \tilde s^2$.
It follows that
 \begin{equation} \label{Eq.probslowsmallbox}
\P \( |\tilde B \cap \tilde \ss | \ge a \tilde s^2 \)  \ge \frac 1 {\tilde s^2} \( \E |\tilde B \cap \tilde \ss | - a \tilde s^2 \)  \ge a .
 \end{equation}

Without loss of generality, we suppose $B=[0,s)^2$. We next partition $B$
into $n^2$ identical $\tilde s$-boxes, from which we pick those of the
form
$[4 a \tilde s, (4 a + 1) \tilde s ) \times
[4 b \tilde s, (4 b + 1) \tilde s ) $, $a,b \in \Z \cap [0, n/4)$, and
  enumerate them arbitrarily as
  $\tilde B_i$, $i = 1 , \cdots, (n/4)^2$. Note that
  $\tilde B_i \cap \tilde \ss $ depends on the restriction
  of the fine field $\psi_{\tilde r}$ to
  the $(2 \tilde s)$-box centered at $c_{\tilde B_i}$,
  and $\psi_{\tilde r} (w)$ is independent of
  $\psi_{\tilde r} (w^\prime)$ if $|w-w^\prime| \ge 2 \tilde s$. It follows
  that the random variables
  $|\tilde B_i \cap \tilde \ss |$'s are mutually independent. Let
 $$
\chi_i = 1 \ \mbox{if } |\tilde B_i \cap \tilde \ss| \ge a \tilde s^2, \ \ \ \chi_i = 0 \mbox{ otherwise}.
 $$
Then $\sum_{i=1}^{n^2/16} \chi_i \ge \e_3 s^2 / (a  \tilde s^2 ) $ implies $|B \cap \tilde \ss|    \ge a \tilde s^2 \times \e_3 s^2 / (a  \tilde s^2 ) = \e_3 s^2$. It follows that
 \begin{equation} \label{Eq.FFLD}
\P (|B \cap \tilde S | \ge \e_3 s^2) \ge \P \( \sum_{i=1}^{n^2/16} \chi_i  \ge \frac {\e_3 s^2}{a \tilde s^2} \) .
 \end{equation}

Now we estimate the right hand side of \eqref{Eq.FFLD} via large deviations.
Note that the $\chi_i$'s are Bernoulli random variables, with $P(\chi_i=1)
\geq a$, see \eqref{Eq.probslowsmallbox}, and therefore
 $$
\E e^{- \chi_i} = 1 -  ( 1 - e^{-1} )  \P (\chi_i = 1)  \le 1 - (1 - e^{-1})  a \le \exp ( - (1 - e^{-1}) a) .
 $$
Using independence and Chebyshev's inequality we get
\begin{equation}
  \label{eq-pg11}
\P \( \sum_{i=1}^{n^2/16} \chi_i < \frac {\e_3 s^2 }{ a  \tilde s^2 } \)  \le \exp \( \frac {\e_3 s^2}{a  \tilde s^2} \) \(  \E  e^{- \chi_1 } \) ^{n^2 / 16} \le  \exp \( \frac {\e_3 s^2}{a  \tilde s^2} - \frac {n^2}{16} (1 - e^{-1}) a \) .
\end{equation}
Recall that
$\tilde s = s / n$, $a = 60 C_3 e^{-6 k \gamma^2}$, see
\eqref{Eq.aforslowpoint}, and $\e_3 \le C_3^2 e^{-12 k \gamma^2} = (\frac a {60} )^2$ by assumption. Thus,
 $$
\frac {\e_3 s^2} {a  \tilde s^2} - \frac {n^2}{16} (1 - e^{-1}) a  \le \( \frac 1 {60^2} - \frac {1-e^{-1}} {16} \)  an^2 \le - 2 C_3 e^{- 6 k \gamma^2} n^2.
 $$
Together with
\eqref{eq-pg11}
and \eqref{Eq.FFLD}, we conclude that
 \begin{equation} \label{Eq.LDforchi}
\P (|B \cap \tilde S| \le \e_3 s^2) \le \P \( \sum_{i=1}^{n^2/16} \chi_i < \frac {\e_3 s^2 }{ a  \tilde s^2 } \)  \le  e^{- 2 C_3 e^{- 6 k \gamma^2} n^2}.
 \end{equation}

\medskip

\noindent {\bf Step 3.} Abbreviate $\s = \s_{z, s}$ and $\tilde \s = \s_{z, \tilde s} $. Recall that $z\in \ss$ if $P^z (F_r (\s) \ge \e_1 s^2) \ge \e_2$ while $z\in \tilde \ss$ if $P^z (F_{\tilde r} (\tilde \s) \ge \tilde \e_1 \tilde s^2) \ge  \e_2$. Since $\tilde s < s$, it holds that $\tilde \s < \s$.
Consequently, $F_r^z (\tilde \s) \le F_r^z ( \s )$. Therefore,
 \begin{equation} \label{Eq.FFcompareFr}
P^z (F_r (\s) \ge \e_1 s^2) \ge P^z (F_r (\tilde \s) \ge  \e_1  s^2 ) , \ \ \ \mbox{for all } z.
 \end{equation}
We are going to compare $F_r^z (\tilde \s)$ with $F_{\tilde r}^z (\tilde \s)$,
and show below that
  \begin{equation} \label{Eq.compareFrFr}
\P \(  \ee \) \ge 1 -  e^{- n^2 \log n }, \ \mbox{where } \ee =
\{P^z (F_r (\tilde \s) \ge \e_1 s^2)
  \ge P^z (F_{\tilde r} (\tilde \s) \ge \tilde \e_1 \tilde s^2)  \; \mbox{\rm for all} \;
z\in B\} .
 \end{equation}
Combined \eqref{Eq.FFcompareFr}, it follows that if $\ee$ occurs then $z\in \tilde \ss \Rightarrow
z\in \ss$, for all $z \in B$, and in particular $\ee \subset \{  B \cap \tilde \ss \subset B \cap \ss \} $. It follows then
from
\eqref{Eq.compareFrFr}  that
 \begin{equation} \label{Eq.FFslowtslow}
\P ( B \cap \tilde \ss  \nsubseteq B \cap \ss)  \le \P (\ee ^c) \le e^{- n^2 \log n} ,
 \end{equation}
 which we will use in the next step. Before doing that, we first complete the
 proof of
\eqref{Eq.compareFrFr}.

Let $\phi = \psi_r  - \psi_{\tilde r}$, which has covariance
 $$
G_{r,\tilde r} (w_1, w_2) = k \log 2 \sum_{j=r}^{\tilde r - 1} A (w_1, w_2; 2^{-kj}) .
 $$
Set
 $$
M = \max_{w \in \breve{B} } ( -\phi (w)) ,
\ \ \ \mbox{where } \breve{B} = [- \frac 1 2 s, \frac 3 2 s )^2 \ \ \mbox{is the $2 s$-box centered at $c_B$.}
 $$
Set $\hat B = 2^{kr} \breve{B}$, which has side length $2$. Note that
$A(w_1, w_2, 2^{-kj}) = A(\hat w_1, \hat w_2, 2^{- k (j-r)})$, where $\hat w_i = 2^{k r} w_i$. Therefore, $\{ \phi (w), w \in \breve{B} \}$ is a copy of the coarse field $\{ \varphi_{r_0} (\hat w), w \in \hat B \}$, with $w$ being identified as $\hat w = 2^{k r} w$, where we recall that $r_0 = \tilde r - r$ and is defined in \eqref{Eq.defofn}. By Corollary~\ref{Cor.maxvalue}, $\E M \le \sqrt 2 C_0 \sqrt {2^{k r_0} \times 2} = 2 C_0 \sqrt {n}$. Since $\E \phi (w)^2 =  k r_0 \log 2 = \log n$ for all $w$, we have
 \begin{equation} \label{Eq.Maxslowbox}
\P ( M \ge 2 n \log n) \le 2 e^{- \frac {(2 n \log n -2  C_0 \sqrt {n} )^2} {2 \log n}} \le  e^{- n^2 \log n }   ,
 \end{equation}
where we use Lemma~\ref{Lem.concentration}, and the last inequality holds by \eqref{Eq.largenslowbox}. Noting for all $z \in B$, the $\tilde s$-box centered at $z$ is contained in $\breve{B}$, we have $X_u \in \breve B$ for $u \le \tilde \s$, where we drop the superscript $z$ in $X_u$. Therefore, on the event $\{ M < 2 n \log n \}$, it holds that
for all $z\in B$,
 \begin{eqnarray*}
F^z_r (\tilde \s)
 & = &
\int_0^{\tilde \s} e^{\gamma \psi_{\tilde r} (X_v) - \frac {\gamma^2} 2 \E \psi_{\tilde r} (X_v)^2 } \times e^{\gamma \phi (X_v) - \frac {\gamma^2} 2 \E \phi (X_v)^2 } d v
 \\ & \ge &
e^{- \gamma M - \frac {\gamma^2} 2 k r_0 \log 2 } F_{\tilde r} (\tilde \s ) \ge e^{- \gamma 2 n \log n - \frac {\gamma^2} 2 \log n } F_{\tilde r} (\tilde \s ) ,
 \end{eqnarray*}
where in the first equality we use the independence of $\psi_{\tilde r}$ and $\phi$. By the definition of $\tilde \e_1$ in \eqref{Eq.tildeepsl},
 $$
P^z (F_r (\tilde \s) \ge \e_1 s^2) \ge P^z (F_{\tilde r} (\tilde \s)  \ge e^{\gamma 2 n \log n + \frac {\gamma^2} 2 \log n } \e_1 s^2 )  = P^z (F_{\tilde r} (\tilde \s)  \ge  \tilde \e_1 \tilde s^2 ).
 $$
Therefore, we conclude that $\{M<2n\log n\}\subset \ee$. This, together with
\eqref{Eq.Maxslowbox}, implies \eqref{Eq.compareFrFr} and completes the proof of
 \eqref{Eq.FFslowtslow}.

\medskip

\noindent {\bf Step 4.} If $|B \cap \tilde S| \ge \e_3 s^2$ and $B \cap \tilde S \subset B \cap S$, we have $|B \cap S| \ge \e_3 s^2$, i.e. $B$ is slow. Hence,
 $$
1 - \P (B \mbox{ is slow}) \le \P (|B \cap \tilde S| \le \e_3 s^2) + \P ( B \cap \tilde S \nsubseteq B \cap S ).
 $$
By \eqref{Eq.LDforchi} and \eqref{Eq.FFslowtslow}, it follows that
 \begin{eqnarray*}
1 - \P (B \mbox{ is slow})
 & \le &
\exp \{ - 2 C_3 e^{- 6 k \gamma^2} n^2 \} + \exp \{ - n^2 \log n \}
 \\ & \le &
\exp \{  - C_3 e^{- 6 k \gamma^2} n^2 \} \le \exp \{  - C_3 e^{- 6 k \gamma^2} 2^{- 2 k} \k^2 \} = \e_1^{C_3 e^{- 6 k \gamma^2} 2^{- 2k}} ,
  \end{eqnarray*}
where in the second inequality we use \eqref{Eq.largenslowbox} and in the last two inequalities we use \eqref{Eq.defofn}. This implies (ii) and completes the
proof of the lemma.
  \end{proof}

The next lemma bounds below $F_r^z (\sigma_{z, 3s})$ uniformly in $z$ in slow boxes.
 \begin{lemma} \label{Lemma.claim2}
There exists a universal positive constant  $C_4$ such that the following holds. Suppose $B$ is slow. Then, $P^z (F_r (  \s_{z,3s}) \ge \e_1 s^2) \ge C_4 \e_2 \e_3$ for all $z$ in the closure of $B$.
 \end{lemma}
 \begin{proof}
Abbreviate $\s^\prime = \s_{z,3s} $.
Let
$\rho_1 (w, w^\prime)$ denote the heat kernel of the SBM, killed
upon exiting $[0,3]^2$, at time $1$. Let $C_4 := \min_{ w, w^\prime \in [0.5, 2.5]^2} \rho_1 ( w, w^\prime )$, which is positive.
Suppose that
the SBM started
from $z$ hits $B \cap \ss$ at time $\s_*$ and point $w$. Since $|B \cap \ss|
\ge \e_3 s^2$, we have that
$P^z (\s_* < \s^\prime) \ge C_4 \e_3$. On $\s_* < \s^\prime$,
$F_r^z (\s^\prime) \ge \s $, where $\s $ is the time that the
$\psi_r$-LBM started
from $w$ exits $B_s (w)$. Since $w\in \ss$,
$P^w (\s \ge \e_1 s^2) \ge \e_2$. By the
strong Markov property, $P^z (F_r (\s^\prime) \ge \e_1 s^2)
\ge P^z (\s_* < \s^\prime, \s \ge \e_1 s^2) \ge C_4 \e_3 \times \e_2$,
which
completes the proof.
  \end{proof}

\section{Lower Bound}\label{Section.lowerbound}

We continue to take
$s := 2^{-kr} = t^{\frac 1 {1 + \frac 1 2 \gamma^2}  + o (1)}$. To obtain the
lower bound on the LHK, we will force the LBM
$\{ Y^x_u \}$, started at $x\in \T$, to hit $y\in \T$
according to the following three steps. First, we will force the LBM
to hit inside $BD_r (y)$ a point which is \textit{very fast} (a notion to be defined below), then hit inside $B(y, s^{1 + \b^\prime})$
(where $\b^\prime > 0$ is a parameter to be chosen), and finally we force the LBM to
hit $y$.
We will allow time about $t/3$ for each step,
and show that these steps respectively bring factors
$e^{-s^{- (1+o(1))}}$, $s^{2 + 2 \b^\prime + o(1)}$ and $O(1)$
for the lower bound of the heat kernel. This will give the
lower bound $e^{ - s^{- (1+o(1))} } s^{2 + 2 \b^\prime + o(1)} $,
which is $\ge \exp ( - t^{- \frac 1 { 1 + \frac 1 2 \gamma^2 }  - \e } )$
as required.

The argument is naturally split according to these steps.
In Subsection~\ref{Section.LBHP}, we compute the probabilities of the
first step in Lemma~\ref{Lemma.hitlastbox} and of
the second one in Lemma~\ref{Lemma.hitsneary}, after introducing the notion
of very fast points; in that section, $r$
will be arbitrary, {\em i.e.} not tied to the value of $t$.
We pick the value of $r$ according to $t$ in Subsection~\ref{Section.LBPF},
where we will deal with the third step and show the lower bound.

\subsection{Lower bound for hitting probability} \label{Section.LBHP}

Suppose $\d > 0$, $r \ge 1$ integer, and set
 $s=2^{-kr}$.
 Take $\d_1 = s^{ 3 \delta}$, $\d_2 = s^{2 \delta}$, $\d_3 = s^{\delta}$,
 and define fast points/boxes with respect to the parameters
 $\d_1$, $\d_2$ and $\d_3$.

 \begin{lemma} \label{Lemma.hitlastbox}
There exist positive constants
$c$, $k_0 = k_0 (\d)$,
$c_0 = c_0 (k, \d)$ and $r_0 = r_0 (x,y, \gamma, \d, k )$,
not depending on $r$ but possibly depending on $k,\gamma$,
such that the following holds for $k \ge k_0$ and $r \ge r_0$.
Suppose $D$ is a random (with respect to $h$) set and $D \subset BD_r (y)$.
Let $\varsigma_1$ be the hitting time of $D$ by the LBM started from $x$.
Then, with $\P$-probability at least $1 - e^{- c_0 r} - \P ( |D| < \d_3 s^2)$,
 \begin{equation} \label{Eq.theta1}
P^x ( \varsigma_1 \le s^{1 + \frac 1 2 \gamma^2 - 4 \d - c \gamma \d } )  \ge e^{- s^{-(1+2 \d ) }} .
 \end{equation}
 \end{lemma}
 \begin{proof}
We construct a sequence of neighboring $s$-boxes connecting $x$ and $y$, as follows. Discretize $\T$ by regarding each $B \in \BD_r$ (equivalently, its center $c_B$) as a point in $\Z^2$. We investigate the discrete Gaussian field $\Phi: = \{ \varphi_r (c_B), B \in \BD_r \}$, together with the Bernoulli process $\Xi : =\{ \xi_B, B \in \BD_r \}$ defined by $\xi_B : = 1$ if $B$ is fast. Next we will apply \cite[Theorem~1.7]{DZ15} to $(\Phi, \Xi)$. Set $N = 2^{kr}$, and correspond $B$, $\varphi_r (c_B)$, $\xi_B$ respectively to $w \in \Z^2$, $\varphi_{N, w}$, $\xi_{N, w}$ in \cite{DZ15}. Then,
 \begin{itemize}
\item $\Xi$ is independent of $\Phi$, since $\Xi$ depends on the fine field while $\Phi$ depends on the coarse field.
\item The collection of random variables $\{\xi_B \}_{B\in \BD_r}$ has finite range dependence, in particular $\xi_B$ is independent of $\xi_{B'} $ if $|c_B-c_{B'}|_\infty > 9 s$. (In the language of \cite{DZ15}, $\Xi$ is $q$-dependent for $q = 9$.)
\item $P(\xi_B = 1)$ is equal to a same value $p$ for all $B$.
 \end{itemize}
For constants $c (\geq 2),\d,r$, we introduce the event $\ee_1 = \ee_1 (c,\d,r,k)$ defined as the existence of a sequence $B_i$,
$i=1,\cdots,I$ of $s$-boxes in $\BD_r$ satisfying the following properties:
 \begin{enumerate}[(a)]
\item $\varphi_r (c_{B_i}) \le (c-1) \d k r \log 2$, $i=1,\ldots,I$.
\item $B_i$ is fast ({\em i.e.}, $\xi_{B_i} = 1$), $i=1,\ldots,I$.
\item $I \le s^{-(1+\d)}$.
\item $B_1 = BD_r (x)$, $B_I = BD_r (y)$, and
  $B_{i+1}$ is a neighbor of $B_i$,
  i.e. $|c_{B_{i+1}} - c_{B_i}| = s$, $i=1,\ldots,I-1$.
 \end{enumerate}
By Lemma~\ref{Lemma.fastbox}, $p \ge 1 - (C_1 + 1) s^{\d} \to 1$ as $r \to \infty$. In particular, $p$ is larger than $p_1$ defined in \cite[Theorem~1.7]{DZ15}, when $r \ge r_1 (\d )$. As in \cite[Theorem~1.7]{DZ15}, there exist positive constants $c (\geq 2)$, $k_0$, $\tilde c_0 = \tilde c_0 (\d)$ and
$r_2 = r_2 (x, y, \gamma, \d, k) \ge r_1$ so that,  for $k \ge k_0$ and $r \ge r_2$,
 \begin{equation}  \label{eq-pg-13}
  \P(\ee_1)
\ge 1 -  (1-p)^{1/400} - e^{- \tilde c_0 r} ,
 \end{equation}
where we use $q = 9$ and $p \to 1$ as $r \to \infty$.

 \begin{remark}
(i) The space is the torus $\T$ here, while it is a box in \cite{DZ15}. One can identify the torus as $[0,4)^2$, and consider the box $[1,3]^2$ where we locate $x$ and $y$, noting that $h(z)$ is independent of $h(w)$ if $|z-w| \ge 2$. (ii) To achieve \eqref{eq-pg-13}, it is not crucial whether one uses balls $B(x, R)$ (as in our situation) or boxes $B_{2R} (x)$ (as in \cite{DZ15}) to define $A(x,y; R)$. That is, the proof of  \eqref{eq-pg-13} is similar 
to that of \cite[Theorem~1.7]{DZ15}.
 \end{remark}

Let $\ee_2$ be the event that the following properties hold.
 \begin{enumerate}[(a$^\prime$)]
\item  $|\varphi_r (z) - \varphi_r (c_B)| \le \d k r \log 2$ for all $z \in B^*$  and $B \in \BD_r$.

\item $x$ is fast.
 \end{enumerate}
By Corollary~\ref{Cor.fluctuation}, $\P(\mbox{a}^\prime)\geq 1-e^{-r}$. By Lemma~\ref{Lemma.fastpoint},  $\P(\mbox{b}^\prime)\geq 1- C_1 \d_1/ \d_2 = C_1 2^{- k \d r}$. Take $c_0$ such that $(C_1+1)^{\frac 1 {400} }2^{- \frac {k \d} {400} r } +
e^{- \tilde c_0 r} + e^{-r } + C_1 2 ^{- k \d r} \le e^{- c_0 r}$. Then, we have
 $$
\P (\ee) \ge 1 - e^{- c_0 r} - \P ( |D| < \d_3 s^2 ), \ \mbox{where } \ee = \ee_1 \cap \ee_2 \cap  \{ |D| \ge \d_3 s^2 \} .
 $$

Next, we are going to show that \eqref{Eq.theta1} holds on $\ee$, completing the proof. Suppose $\ee$ holds. We will force the SBM to follow this sequence of boxes; to control
the LBM time, we will force also passage through
fast points, and some additional properties,
as follows.
Recall that $\{ X^x_u \}$ is the SBM starting from $x$. Construct a sequence of
hitting times $\s_i$ as follows.
Let $\s_1 = 0$. Then $X_{\s_1}^x = x \in B_1 \cap \ff$ by (b$^\prime$).
Suppose that $\s_i$ has been defined,
such that $x_i : = X^x_{\s_i} \in \cB_i \cap \ff $. Define
 $$
\s_{i+1} : = \inf\{ u \ge \s_i : X^x_u \in A \} , \ \ \ \mbox{and } \tau_i = \s_{i+1} - \s_i, \ \ \ \mbox{where } A = \left\{ \begin{array}{ll} \cB_{i+1}  \cap \ff , & \mbox{ if } i \le I-2 , \\
  D, & \mbox{ if } i = I - 1. \end{array} \right.
 $$
Informaly, $\tau_i$ is the time
it takes for the SBM to cross
$B_i$ into the next box $B_{i+1}$ and hit a fast point.

Note that (a) together with (a$^\prime$) implies that
  \begin{enumerate}[(a$^{\prime \prime}$)]
\item  For all
  $z \in \cup_i B_i^{*}$,
  $\varphi_r (z) \le c \d k r \log 2$.
 \end{enumerate}
In order to take advantage of (a$^{\prime\prime}$), we need to also control
the path of the SBM when traveling from
$x_i$ to $B_{i+1}\cap \ff$. Toward this end,
define
 $$
\tilde \s_i = \inf\{ u \ge \s_i : X^x_u \in \partial B_i^{*} \} \ \ \ \mbox{and } \tilde \tau_i = \tilde \s_i - \s_i .
 $$
Thus, $\tilde \tau_i$ is the time it
takes the SBM to exit $B^*_i$ when starting at $x_i$.
We will force the events
$\tau_i \le s^2$ and $\tau_i \le \tilde \tau_i$ to ensure that
the LBM stays inside $B_i^{*}$
and spends a short enough time to hit $B_{i+1} \cap \ff $.

Let $\rho_1 (w, w^\prime)$ denote
the heat kernel of  the SBM, killed at exiting
$[0,5]^2$, at time $1$.
Let
 \begin{equation} \label{Eq.defC5}
C_5 : = \frac 1 2 \min_{w, w^\prime \in [1,4]^2} \rho_1 (w, w^\prime) ,
 \end{equation}
which is positive. Then, for any $i \ge 1$,
 $$
P^x (\tau_i \le s^2 \le \tilde \tau_i ) \ge 2 C_5 \d_3
 $$
since on $\ee$, $| \cB_{i+1} \cap \ff | \ge \d_3 s^2$ by (b), and $|D| \ge \d_3 s^2$. Let
 $$
\hat \tau_i : = \inf \{ u \ge 0 : X_{\s_i + u}
\in \partial B_{6 s} (x_i) \} .
 $$
Recall
that $x_i$ is a fast point, $\forall i \le I-1$.  By the
strong Markov property of the $\psi_r$-LBM,
 $$
P^x (F_r (\s_i + s^2 \wedge \hat \tau_i) - F_r (\s_i) \le s^2 / \d_1) = P^{x_i} (F_r (s^2 \wedge \s_{x_i, 6 s}) \le s^2 / \d_1 ) \ge 1 - \d_2 .
 $$
Therefore,
 $$
P^x (\tau_i \le s^2 \le \tilde \tau_i , \ F_r (\s_i + s^2 \wedge \hat \tau_i) - F_r (\s_i) \le s^2 / \d_1) \ge 2 C_5 \d_3 - \d_2 \ge C_5 \d_3
 $$
for $r$ larger than $r_3 : = r_3 (x,y, \gamma, \d , k) \ge r_2$, where we used that
$\d_2 = o (\d_3)$ as $r \to \infty$.
By definition, $\tilde \tau_i \le \hat \tau_i$.
Hence, if $\tau_i \le s^2 \le \tilde \tau_i$,
we have $\tau_i \le s^2 \wedge \hat \tau_i$ thus
$F_r (\s_{i+1}) \le F_r (\s_i + s^2 \wedge \hat \tau_i) $, and by (a$^{\prime\prime}$),
 $$
F^x (\s_{i+1}) - F^x (\s_i) \le e^{\gamma c \d k r \log 2 - \frac 1 2 \gamma^2 k r \log 2}  \big( F^x_r (\s_{i+1}) - F^x_r (\s_i) \big).
 $$
Collecting the above inequalities, we have that for $i=1,\ldots,I-1$,
 \begin{equation} \label{Eq.throughabox}
P^x (F (\s_{i+1}) - F (\s_i) \le e^{\gamma c \d k r \log 2 - \frac 1 2 \gamma^2 k r \log 2} s^2 / \d_1)  \ge  C_5 \d_3 .
 \end{equation}
Finally, note that $\varsigma_1 \le \sum_{i=1}^{I-1} ( F^x (\s_{i+1}) - F^x (\s_i ) )$. By (c), \eqref{Eq.throughabox} and the strong Markov property of the LBM,
 \begin{equation} \label{Eq.LBevent}
P^x ( \varsigma_1 \le |I| e^{\gamma c \d k r \log 2 - \frac 1 2 \gamma^2 k r \log 2} s^2 / \d_1) )  \ge  (C_5 \d_3)^{|I|}  \ge e^{- s^{ -(1+ 2 \d) }}
 \end{equation}
for $r \ge r_0 \ge r_3$. Note however
that
$ |I| e^{\gamma c \d k r \log 2 - \frac 1 2 \gamma^2 k r \log 2}
s^2 / \d_1 \le s^{1 + \frac 1 2 \gamma^2 - 4 \d - c \gamma \d }$. Together with
\eqref{Eq.LBevent},
this completes the proof of the lemma.
 \end{proof}

Let $\b^\prime > 0$ be fixed. Abbreviate $B = BD_r (y)$, and set $A = B \cap B(y, s^{1 + \b^\prime} )$.
Denote by $\tau_A$ (respectively, $\tau^*$) the times that the SBM hits
$A$ (respectively,
$\partial B^{*}$). A point $z \in B$ is called {\em very fast}
if $P^z ( F_r(s^2) \le s^{2 - \d} | \tau_A \le s^2 \le \tau^* ) \ge 1 / 2$.
Let $\vff$ denote the set of very fast points. Note that
$\vff \subset B$. We would like to mention that the very fast property does not imply the fast property.

\begin{lemma} \label{Lemma.hitsneary}
(i) $\P ( |\vff| \ge \d_3 s^2 ) \ge 1 - 3 s^\d$.\\
(ii) Let $\varsigma_2$ denote
the time that the LBM hits $A$. Then,
there exists $r_1 = r_1 (\d, \gamma, k)$
such that the following holds for $r \ge r_1$.
With $\P$-probability at least $1 - 2 e^{- \frac 1 8 \d^2 k r \log 2}$,
 \begin{equation} \label{Eq.theta2}
P^z ( \varsigma_2 \le s^{2 + \frac 1 2 \gamma^2 - \d - \gamma \d } ) \ge s^{2 + 2 \b^\prime + \d }, \ \ \ \forall z \in \vff .
 \end{equation}
 \end{lemma}
 \begin{proof}
The proof of (i) is parallel to Lemma~\ref{Lemma.fastbox}(i) combined with Lemma~\ref{Lemma.fastpoint}(i), while that of (ii) is parallel to \eqref{Eq.throughabox}.

(i) Set $\xi = F_r^z (s^2) $ and
$\eta = P^z ( \xi > s^{2-\d} | \tau_A \le s^2 \le \tau^* )$.
By a proof similar to that of  Lemma~\ref{Lemma.fastpoint}(i),
$\P (z \notin \vff ) = \P (\eta > 1/2)  \le 2 \E \eta = 2 E^z \big( \P(\xi > s^{2-\d}) | \tau_A \le s^2 \le \tau^* \big) \le 2 s^\d$ since $\P (\xi > s^{2 - \d }) \le s^{\d -2} \E \xi = s^\d$, for all $z \in B$.
Then, $(1 - 2 s^{\d} ) s^2 \le \E |\vff| \le s^2 \P (|\vff| \ge \d_3 s^2)+ \d_3 s^2$, {\em i.e.} $ \P (|\vff| \ge \d_3 s^2) \ge 1 - 2 s^\d - \d_3 = 1 - 3 s^\d$, where we recall that $\d_3 = s^\d$. 

(ii) For any $z \in \vff$,
 $$
P^z (F_r (s^2) \le s^{2-\d},
\tau_A \le s^2 \le \tau^*  ) \ge \frac 1 2 P^z ( \tau_A \le s^2 \le \tau^*  ).
 $$
With $C_5$ defined in \eqref{Eq.defC5}, we have $P^z ( \tau_A \le s^2 \le \tau^*  )  \ge 2 C_5 |A| \ge 2 C_5 \times \frac 1 4 \pi  s^{2(1 + \b^\prime)}$. It follows that, for $r$ large enough,
$$  P^z (F_r (s^2) \le s^{2 - \d}, \tau_A \le s^2 \le \tau^*  ) \ge
\frac {C_5 \pi}{4} s^{2 + 2 \b^\prime} \ge s^{2 + 2 \b^\prime + \d}.$$
By Corollary~\ref{Cor.fluctuation}, with probability $\ge 1 - 2 e^{- \frac 1 8 \d^2 k r \log 2}$, we have $\varphi_r (w) \le \d k r \log 2$ for all $w \in B^*$. On this event,
 $$
\{ F_r (s^2) \le s^{2-\d}, \tau_A \le s^2 \le \tau^* \}
\Rightarrow \{\varsigma_2^z \le e^{\gamma \d k r \log 2 - \frac 1 2 \gamma^2 k r \log 2} s^{2 - \d}\}
 $$
for all $z \in B$. Noting that
$e^{\gamma \d k r \log 2 - \frac 1 2 \gamma^2 k r \log 2} s^{2 - \d}= s^{2 + \frac 1 2 \gamma^2- \d - \gamma \d}$ completes the proof.
 \end{proof}

\subsection{Proof of the lower bound in \eqref{eq-main1}} \label{Section.LBPF}
We take
 $$
r_t = \lceil - \frac {\log t - \log 3}{(1 + \frac 1 2 \gamma^2 - 4 \d - c \gamma \d) k \log 2 } \rceil,
 $$
and set $s = 2^{- k r_t}$ so that
 \begin{equation} \label{Eq.tands}
2^{-k} (t/3)^{\frac 1 { 1 + \frac 1 2 \gamma^2 - 4 \d - c \gamma \d }} < s \le (t/3)^{\frac 1 { 1 + \frac 1 2 \gamma^2 - 4 \d - c \gamma \d }} .
 \end{equation}

 The following lemma is a straight forward adaptation of \cite[Corollary 5.20]{MRVZ14}. We omit the details.
 \begin{lemma} \label{Lemma.nearneighbor}
There exists a constant $\b  = \b (\gamma, k)$ and a positive
random variable $ U_0 = U_0 (\gamma, k; h)$ such that for all
$u \le U_0$,
 $$
\inf_{ z\in \T} \ \inf_{w\in \T, |w-z| \le u^\b} p_u^\gamma (z,w) \ge 1.
 $$
 \end{lemma}

Set $\b^\prime = (1 + \frac 1 2 \gamma^2 - 4 \d - c \gamma \d) \b$. By \eqref{Eq.tands}, $ \ell : = s^{1 + \b^\prime} \le s^{\b^\prime} \le s^{(1 + \frac 1 2 \gamma^2 - 4 \d - c \gamma \d) \b }  \le (t/3)^\b$. Let $\varsigma$ be the time the LBM hits the small ball $B(y, \ell)$. On the event $\varsigma \le 2t/3$, $u: = t-\varsigma \ge t/3$. It follows $\ell \le u^\b$. Consequently, by strong Markov property and Lemma~\ref{Lemma.nearneighbor}, it follows
 \begin{equation} \label{Eq.LBhtprob}
p_t^\gamma (x,y) \ge P^x (\varsigma \le 2 t/3) , \ \ \ \forall t \le U_0.
 \end{equation}

Next, we estimate $P^x ( \varsigma \le 2 t/3)$. We follow the notations in  Lemma~\ref{Lemma.hitlastbox} and Lemma~\ref{Lemma.hitsneary}. Define very fast points with respect to the parameter $\b^\prime$, and take $D$ as $\vff$. Then, for any $r \ge r_0 \vee r_1$,  \eqref{Eq.theta1} and \eqref{Eq.theta2} hold simultaneously, with probability $1 - e^{- c_0 r} - 3 s^\d - 2 e^{- \frac 1 8 \d^2 k r \log 2}$. Note that
$t \to 0$ is equivalent to $r_t \to \infty$. By the
Borel-Cantelli Lemma, we can find $T_0 = T_0 (x,y, \gamma, \e, k; h)<U_0 $ such that for all $t \le T_0$,
both \eqref{Eq.theta1} and \eqref{Eq.theta2} hold for $r = r_t$,
and furthermore
 \begin{equation} \label{Eq.tands1}
e^{- s^{-(1+ 2 \d )}}  s^{2 + 2 \b^\prime + \d }  \ge \exp \( - t^{ - \frac 1 { 1 + \frac 1 2 \gamma^2 }  - \e } \)
 \end{equation}
where we take $\d$ (according to $\e$) such that $\frac {1 + 2 \d }{1+\frac 1 2 \gamma^2 - 4 \d - c \gamma \d} < \frac 1 {1 + \frac 1 2 \gamma^2} + \e$.
By the
strong Markov property, $P^x (\varsigma \le 2t/3) \ge P^x (\varsigma_1 \le t/3) \min_{z \in \vff} P^z (\varsigma_2 \le t/3) \ge e^{- s^{-(1+3\d )}}  s^{2 + 2 \b^\prime + \d }$. This, together with \eqref{Eq.LBhtprob} and \eqref{Eq.tands1}, gives the lower bound
in \eqref{eq-main1}.
\qed

\section{Proof of the upper bound in \eqref{eq-main1}}
\label{Section.upperbound}

We begin with the following lemma, whose proof is a slight adaptation of
that of \cite[Theorem 4.2]{MRVZ14}. We omit further details of the proof.
 \begin{lemma} \label{Lem.coarseub2}
For any $\e>0$ there exist $\b = \b (\e, \gamma, k) > 0$ and
positive random constants $c_1 = c_1 (h)$ and $c_2 = c_2 (h)$ such that,
for all $z,w\in \T$ and $u>0$,
 $$
p_u^\gamma (z, w) \le \frac {c_1 }{u^{1+\e}} \exp \( - c_2   \( \frac {| z-w |}{u^{1/\b }} \) ^{\frac {\b }{\b -1} } \) .
 $$
 \end{lemma}
 We turn to the proof of the upper bound in \eqref{eq-main1}.
Fix $\a$ such that
  $$
\a > 1 \  \mbox{ and } \ (\frac \a {\b } - 2) \frac {\b }{\b - 1} \ge \frac 1 {1 + \frac 1 2 \gamma^2} \ ,
 $$
and set $u = t^\a$ in Lemma~\ref{Lem.coarseub2}. Then,
for $z \notin B(y, t^2)$,
 $$
p_{t^\a}^\gamma (z, y) \le \frac {c_1 }{t^{\a (1+\e) }} \exp \( - c_2   \( \frac {t^2}{t^{\a /\b }} \) ^{\frac {\b}{\b -1} } \) \le \frac {c_1 }{t^{\a (1+\e)}} \exp \( - c_2 t^{- \frac 1 {1 + \frac 1 2 \gamma^2 }} \) \le \exp \( - t^{- \frac 1 {1 + \frac 1 2 \gamma^2 } + \frac 1 2 \e } \) ,
 $$
where the last inequality holds for $t$ smaller than some
$T_1 (\gamma, \e, k, h)$. It follows that
 \begin{equation} \label{Eq.upoffsmallball}
\int_{|z - y| \ge t^2}  p_{t -t^\a}^\gamma (x,z) p_{t^\a}^\gamma (z,y)
\mu^\gamma (d z)
 \le
\exp ( - t^{- \frac 1 {1 + \frac 1 2 \gamma^2} + \frac 1 2 \e  }) .
  \end{equation}
On the other hand, again from
Lemma
 \ref{Lem.coarseub2},
$p_{t^\a}^\gamma (z,y) \le \frac {c_1} {t^{\a (1+\e) }}$ for all $z$. Thus,
 $$
\int_{|z - y| < t^2}  p_{t -t^\a}^\gamma (x,z) p_{t^\a}^\gamma (z,y) \mu^\gamma (d z) \le \frac {c_1} { t^{ \a (1+\e)} }
P^x \( | Y_{t - t^\a} - y | < t^2 \) .
 $$

Assume
  $t^2 \le |x-y| / 2$ and
 set $$
\varsigma : = \inf \{ u \ge 0 : Y^x_u \notin B(x, |x-y|/2 ) \} .$$
 Note that $ \{| Y_{t - t^\a} - y | < t^2 \}\Rightarrow \{\varsigma \le t\} $.
In Lemma~\ref{Lemma.hitprobub} below, we will show
 \begin{equation} \label{Eq.UBexitprob}
P^x (\varsigma \le  t ) \le \exp \( - t^{ - \frac {1}{1 + \frac 1 2 \gamma^2} + \frac 1 2 \e  } \)
 \end{equation}
for $t$ smaller than some $T_2 (\gamma, k, \e; h)$. It then follows that
 $$
\int_{|z - y| < t^2}  p_{t -t^\a}^\gamma (x,z) p_{t^\a}^\gamma (z,y) \mu^\gamma(d z) \le \frac {c_1} { t^{\a (1+ \e )}} \exp ( - t^{- \frac 1 {1 + \frac 1 2 \gamma^2} + \frac 1 2 \e  } ) .
 $$
Combining the above inequality with \eqref{Eq.upoffsmallball}, we conclude that
 \begin{eqnarray*}
p_t^\gamma (x,y) & = & \int p_{t -t^\a}^\gamma (x,z) p_{t^\a}^\gamma (z,y) \mu^\gamma (d z)
 \\ & = &
\int_{|z - y| < t^2}  p_{t -t^\a}^\gamma (x,z) p_{t^\a}^\gamma (z,y) \mu^\gamma (d z) + \int_{|z - y| \ge t^2}  p_{t -t^\a}^\gamma (x,z) p_{t^\a}^\gamma (z,y) \mu^\gamma (d z)
 \\ & \le &
(1 + \frac {c_1}{t^{\a (1+\e)}}) \exp (t^{- \frac 1 {1 + \frac 1 2 \gamma^2} + \frac 1 2 \e })  \le  \exp (t^{- \frac 1 {1 + \frac 1 2 \gamma^2} + \e })
 \end{eqnarray*}
for $t $ less than some $T_0$. This completes
the proof of
the upper bound in \eqref{eq-main1}, modulu the proof of
Lemma \ref{Lemma.hitprobub}.
\qed

 \begin{lemma} \label{Lemma.hitprobub}
There exists $k_0 = k_0 (\e)$ and a random variable $T_2=T_2(\gamma,k,\e;h)$
such that, for all  $k \ge k_0$ and $t<T_2$,
\eqref{Eq.UBexitprob} holds, $\P$-a.s.
  \end{lemma}
 \begin{proof}
The proof is similar to that of Lemma~\ref{Lemma.hitlastbox}.
We will discretize $\T$ using $\BD_r$, and show that for $\d > 0$ and $k$ large enough,
 \begin{equation} \label{Eq.ubrlarge}
P^x ( \varsigma \le  2^{- kr( 1 + \frac 1 2 \gamma^2  + 3 \d + c \gamma \d )} ) \le e^{-2^{kr(1-2\d)}}
 \end{equation}
for all $r \ge r_0 (\gamma, k, \d ; h)$, $\P$-a.s., where $c>0$ is a constant. Then, we will pick a proper $\d$ (according to $\e$) and a proper $r$ (according to $t$), to obtain the lemma.

We begin by discretizing $\T$, fixing
$r \ge 1$ and $s=2^{-kr}$. We identify each $B \in \BD_r$
(equivalently, its center $c_B$) as a point in $\Z^2$ in the natural
way.
We next define inductively
the discrete path associated with the path
$\{ X_u : u \le \tilde \varsigma \}$, where $\{ X_u \}$ is
the SBM starting from $x$ and $\tilde \varsigma$ is the time
$\{ X_u \}$ hits $\partial B(x, \frac 1 4 |x-y|)$. We use the radius $\frac 1 4 |x-y|$ rather than $\frac 1 2 |x-y|$ for the convenient that we do not involve the last point in the discrete path (defined below) to $\partial B(x, \frac 1 2 |x-y|)$.

Let $\tau_1 = 0$. Suppose $\tau_i$ has been defined.
Set $B_i := BD_r (X_{\tau_i})$. Then, define
$$\tau_{i+1} : = \inf \{ u \ge \tau_i : X_u \in \partial B_i^{*} \}.$$
This procedure stops naturally when $\tau_{i+1}$ cannot be defined.
We call this sequence of $B_i$'s a discrete path from $x$ to $\partial B(x, \frac 1 4 |x-y|)$.

Next, set $\e_1 := s^\d $, $\e_2 := C_3 e^{- 6 k \gamma^2}$, $\e_3 := C_3^2 e^{- 12 k \gamma^2}$, and define slow points/boxes with respect to $\e_1$, $\e_2$ and $\e_3$. Set
$\xi_B : = {\bf 1}_{B \ \mbox{is slow}}$. We study the discrete Gaussian field $\Phi = \{ \varphi_r (c_B), B \in \BD_r \}$ and the Bernoulli process $\Xi = \{ \xi_B, B \in \BD_r \}$.
Note that $\Xi$ is of finite range dependence (4-dependent in the language
of  \cite{DZ15}), and by Lemma~\ref{Lemma.fastbox},
$P(\xi_B = 1) = p \ge 1 - 2^{- r k \d C_3 e^{- 6 k \gamma^2} 2^{- 2k}}$,
which converges to $1$ as $r \to \infty$.
For $(\Phi, \Xi)$, similarly to \cite[Theorem~1.5]{DZ15}, we can find positive constants $c$,
$k_0$, $\tilde c_0 = \tilde c_0 (\d)$ and
$r_1 = r_1 (x,y,\gamma, \d, k)$ such that the following holds for
$k \ge k_0$ and $r \ge r_1$. With probability
$\ge 1 - e^{- \tilde c_0 r}$, we can find boxes $B_{i_j}$, $j=1, \cdots, I$ in any discrete path from $x$ to $\partial B(x, \frac 1 4 |x-y|)$ such that $\varphi_r (c_{B_{i_j}}) \ge - (c-1) \d k r \log 2$, $\forall j$, and the following properties hold.
 \begin{enumerate}[(a)]
\item $B_{i_j}$ is slow ({\em i.e.} $\xi_{B_{i_j}} = 1$), $\forall j$.
\item $I \ge s^{-(1-\d)}$.
 \end{enumerate}

Furthermore, by Corollary~\ref{Cor.fluctuation}, with probability at least $1 - e^{-\tilde c_0 r} -e^{-r}$, we have (a), (b) and the following property (c) all hold.
 \begin{enumerate}[(c)]
\item  $\varphi_r (z) \ge - c \d k r \log 2$, $\forall z \in B_{i_j}^{*}$, $\forall j$.
\end{enumerate}

 \begin{remark}
When a discrete path is identified as a sequence of points $v_0, v_1, \cdots$ on $\Z^2$, $v_{i+1}$ may not be a neighbour of $v_i$. However, we have $|v_{i+1} - v_i|_\infty \le 2$ for all $i$. Then, the proof in \cite[Theorem~1.5]{DZ15} automatically extends to the current setup.
 \end{remark}

Set $\s_j = F^x_r (\tau_{i_j+1}) - F^x_r (\tau_{i_j})$ and $\chi_j := {\bf 1}_{\s_j \ge \e_1 s^2}$.
By (a) and Lemma~\ref{Lemma.claim2}, $P^x (\chi_j = 1) \geq C_4 \e_2 \e_3$
for all $j$, which implies that $\E e^{- \chi_j} \le 1 - C_4
\e_2 \e_3 (1 - e^{-1}) \le e^{- C_4 \e_2 \e_3 (1 - e^{-1})}$. Note that the
$\s_j$'s are mutually independent by the strong Markov
property of the $\psi_r$-LBM, and so are the $\chi_j$'s. Therefore,
 \begin{equation} \label{Eq.LDubchi}
P^x \( \sum_{j = 1}^I \chi_\ell  \le \e_1 I \) \le ( e^{\e_1} \E e^{- \chi_j} )^I \le e^{- ( C_4 \e_2 \e_3 (1-e^{-1}) - \e_1) I} \le e^{- \frac 1 2 C_4 \e_2 \e_3 I },
 \end{equation}
where we use that
$\e_1 = 2^{- kr \d} < C_4 \e_2 \e_3 (1 - e^{-1} - \frac 1 2) $ for all
$r$ larger than some $r_2 : =  r_2 (\gamma, \d) > r_1$. By (c), $\chi_j = 1$ implies that
 $$
F^x (\tau_{i_j+1}) - F^x (\tau_{i_j}) \ge  e^{- \gamma c \d k r \log 2 - \frac 1 2 \gamma^2 k r \log 2 } \s_j \ge  2^{- \gamma c \d k r - \frac 1 2 \gamma^2 k r } \e_1 s^2 .
 $$
Thus, $\sum_{j=1}^I \chi_j > \e_1 I$ implies that
 $$
\varsigma >  2^{- \gamma c \d k r  - \frac 1 2 \gamma^2 k r  } \e_1 s^2 \times \e_1 I \ge 2^{- kr( 1 + \frac 1 2 \gamma^2  + 3 \d + c \gamma \d )} .
 $$
This, together with \eqref{Eq.LDubchi} implies that
 $$
P^x ( \varsigma \le  2^{- kr( 1 + \frac 1 2 \gamma^2  + 3 \d + c \gamma \d )} ) \le  \P \( \sum_{j=1}^I \chi_j \le \e_1 I \) \le e^{- \frac 1 2 C_4 \e_2 \e_3 2^{k r(1-\d)}} \le e^{-2^{kr(1-2\d)}} ,
 $$
for all $r$ larger than some $r_3 : = r_3 (\gamma, \d, k) \ge r_2$. By the
Borel-Cantelli Lemma, there exists a random number $r_0 = r_0  (\gamma, k, \d; h)$ such that \eqref{Eq.ubrlarge} holds for all $r \ge r_0$, $\P$-a.s..

For any $t$, define
 $$
r_t : = \lfloor - \frac { \log t  } { \(  1 +  \frac 1 2 \gamma^2 + 3 \d + c \gamma \d \) k  \log 2  } \rfloor .
 $$
equivalently,
 \begin{equation} \label{Eq.ubrelationrt}
2^{k r_t } \le t^{- \frac 1 {1 +  \frac 1 2 \gamma^2 + 3 \d + c \gamma \d}} < 2^{k(r_t +1)} .
 \end{equation}
Note that $t \to 0$ is equivalent to $r_t \to \infty$. Therefore, there exists a random constant $\tilde T_0 = \tilde T_0 (\gamma, k, \d; h)$ such that for any $t \le  \tilde T_0$ (equivalently, $r_t \ge r_0$), \eqref{Eq.ubrlarge} holds for $r = r_t$. This together with \eqref{Eq.ubrelationrt} yields that
 $$
P^x ( \varsigma \le t ) \le \exp \( - (2^{-k} t^{- \frac 1 { 1 +  \frac 1 2 \gamma^2 + 3 \d + 5 \gamma \d } } )^{1 - 2 \d} \) .
 $$
Finally, we pick $\d$ such that  $ \frac {1 - 2 \d} { 1 +  \frac 1 2 \gamma^2 + 3 \d + 5 \gamma \d } > \frac 1 {1 + \frac 1 2 \gamma^2 }   - \frac 1 2 \e $, and then pick $T_0 (\gamma, k, \e; h) \le \tilde T_0$ such that the right hand side above is less than $\exp ( - t^{-  \frac 1  { 1 +  \frac 1 2 \gamma^2 } + \frac 1 2 \e } ) $, completing the proof.
   \end{proof}

\end{document}